\documentclass[12pt]{article}
\usepackage{mathrsfs}
\usepackage{epic,eepic,epsf,epsfig}
\usepackage{amsfonts,srcltx,mathrsfs}
\usepackage{multirow}
\usepackage{amsmath}
\usepackage{amssymb}
\usepackage{amsbsy}
\usepackage{graphicx}
\usepackage{amsfonts}
\usepackage{color}
\usepackage{setspace}
\newtheorem{lem}{Lemma}[section]%
\newtheorem{theorem}[lem]{Theorem}%

\def\nd{\mathrel{\bigm|\kern-.7em/}}

\def\f{\noindent}

\def\P\GammaL{\hbox{\rm P\GammaL}}

\def\mod{\hbox{\rm mod }}

\begin{document}
\title{Extremal graphs for disjoint union of stars and paths}

\footnotetext{E-mails: zhangwq@pku.edu.cn}

\author{Wenqian Zhang\\
{\small School of Mathematics and Statistics, Shandong University of Technology}\\
{\small Zibo, Shandong 255000, P.R. China}}
\date{}
\maketitle

\begin{abstract}
Let $F$ be a graph. A graph $G$ is called $F$-free if $G$ does not contain $F$ as a subgraph. Let ${\rm EX}(n,F)$ denote the set of $F$-free graphs of order $n$ with the maximum edges. In this paper, we characterize the graphs in ${\rm EX}(n,F)$ for large $n$, where $F$ is the disjoint union of paths and stars. This generalizes a result in \cite{LLP}.

\bigskip

\f {\bf Keywords:} Tur\'{a}n number; extremal graph; disjoint union of paths and stars.

\end{abstract}

\baselineskip 17 pt

\section{Introduction}

All graphs considered in this paper are finite, undirected and simple. For a graph $G$, let $\overline{G}$ denote its complement. The vertex set and edge set of $G$ are denoted by $V(G)$ and $E(G)$, respectively. Let $e(G)=|E(G)|$. For two disjoint subsets $S,T\subseteq V(G)$, let $e_{G}(S,T)$ denote the number of edges of $G$ between $S$ and $T$. Let $G[S]$ denote the subgraph of $G$ induced by $S$, and let $G-S=G[V(G)-S]$.  For a vertex $u$, let $d_{G}(u)$ denote the degree of $u$. Also denote $S$ by $u$ if $S=\left\{u\right\}$.  For a certain integer $n$, let $K_{n}, C_{n}$ and $P_{n}$ denote the complete graph, the cycle and the path of order $n$, respectively. 
For an integer $\ell\geq2$, let $\cup_{1\leq i\leq \ell}G_{i}$ denote the disjoint union of graphs $G_{1},G_{2},...,G_{\ell}$. Let $G_{1}\vee G_{2}$ denote the join of $G_{1}$ and $G_{2}$, which is obtained by connecting each vertex of $G_{1}$ to each vertex of $G_{2}$. Also denote by $\ell\cdot H=\cup_{1\leq i\leq \ell}G_{i}$ if $G_{i}=H$ for any $1\leq i\leq \ell$. 

 For a graph $F$, we say that $G$ is $F$-free if $G$ does not contain $F$ as a subgraph. Let ${\rm EX}(n,F)$ denote the set of $F$-extremal graphs of order $n$ (i.e., the $F$-free graphs of order $n$ with the maximum edges). The  Tur\'{a}n number of $F$,  ${\rm ex}(n,F)$ is the size of graphs in ${\rm EX}(n,F)$.  It is one of the cornerstones of graph theory to determine
${\rm ex}(n,F)$ or characterize graphs in ${\rm EX}(n,F)$. This problem can be traced back to 1907, when Mantel (see, e.g., \cite{B}) determined the Tur\'{a}n number of $K_{3}$. In 1940, Tur\'{a}n \cite{T}  generalized this
result by showing that the unique extremal graph for $K_{r}$ is the Tur\'{a}n graph (i.e.,  the complete $(r-1)$-partite graph with balanced parts).
Erd\H{o}s, Stone and Simonovits \cite{E Simonovits,E Stone} proposed the stability theorem
$${\rm ex}(n,F)=\left(1-\frac{1}{\chi(F)-1}\right)\frac{n^{2}}{2}+o(n^{2}),$$
where $\chi(F)$ denote the chromatic number of $F$. When $F$ is non-bipartite, this stability theorem determine approximatively the  Tur\'{a}n number of $F$. Recall that the magnitude $o(n^{2})$ for bipartite graphs $F$ can be improved (see, e.g., \cite{KST}).

As early extremal
results in graph theory, Erd\H{o}s and Gallai \cite{EG} proved the following theorem. 

\begin{theorem}{\rm (\cite{EG})}\label{ex n Pt}
For any $\ell\geq2$ and $n\geq \ell$, let $G$ be a $P_{\ell}$-free graph of order $n$. Then $e(G)\leq\frac{\ell-2}{2}n$ with equality if and only if $n=(\ell-1)t$ and $G=t\cdot K_{\ell-1}$. 
\end{theorem}

Faudree and Schelp \cite{FS} improved this result as follows:

\begin{theorem}{\rm (\cite{FS})}\label{ex n Pt sharp}
Assume that $\ell\geq2$ and $n=(\ell-1)t+s$, where $0\leq s<\ell-1$. let $G$ be a $P_{\ell}$-free graph of order $n$. Then $e(G)\leq t\binom{\ell-1}{2}+\binom{s}{2}$ with equality if and only if either $G=(t\cdot K_{\ell-1})\cup K_{s}$, or $G=\left((t-t_{0})\cdot K_{\ell-1}\right)\cup\left(K_{\frac{\ell}{2}-1}\vee\overline{K_{(\ell-1)t_{0}-\frac{\ell}{2}+s+1}}\right)$ for some $1\leq t_{0}\leq t$ when $\ell$ is even and $s=\frac{\ell}{2}-1$ or $\frac{\ell}{2}$. 
\end{theorem}

 Gorgol \cite{G} gave constructions for the lower bound of the Tur\'{a}n number of $k\cdot P_{3}$. Bushaw and Kettle \cite{BuK} determined the Tur\'{a}n number and
extremal graph of large order for $k\cdot P_{\ell}$ with $\ell\geq3$. For $\ell=3$, their result reads as follows:

\begin{lem}{\rm (\cite{BuK})}\label{kP3}
For integers $k\geq2$ and $n\geq7k$, ${\rm ex}(n,k\cdot P_{3})=\binom{k-1}{2}+(k-1)(n-k+1)+\lfloor\frac{n-k+1}{2}\rfloor$. The unique extremal graph is obtained from $K_{k-1}\vee\overline{K_{n-k+1}}$ by adding a maximum matching in the part $\overline{K_{n-k+1}}$.
\end{lem}

 Lidick\'{y}, Liu and Palmer \cite{LLP} extended
Bushaw and Kettle's result \cite{BuK} as follows. Recall that the Tur\'{a}n numbers of union of specified paths are also studied (see \cite{BieK,LQS,YZ0,YZ}) when $n$ is small. 

\begin{theorem}{\rm (\cite{LLP})}\label{Path union}
For integers $\ell_{1}\geq\ell_{2}\geq\cdots\geq\ell_{m}\geq2$, let $F=\cup_{1\leq i\leq m}P_{\ell_{i}}$, where $m\geq2$ and at least one of $\ell_{1},\ell_{2},...,\ell_{m}$ is not $3$.
Then for large $n$, ${\rm ex}(n,F)=(k-1)(n-k+1)+c$, where $k=\sum_{1\leq i\leq m}\left\lfloor\frac{\ell_{i}}{2}\right\rfloor$, and $c=1$ if all $\ell_{i}$ are odd and $c=0$ otherwise. Moreover, the extremal graph is uniquely $K_{k-1}
\vee(K_{2}\cup \overline{ K_{n-k-1}})$ if all $\ell_{i}$ are odd,
 or $K_{k-1}
\vee\overline{K_{n-k+1}}$ otherwise.
\end{theorem}

Lidick\'{y}, Liu and Palmer \cite{LLP} 
characterized the extremal graphs for union of stars and paths (of special form) as follows. Here, let $K_{1,d}$ denote a star of order $d+1$.  A nearly $k$-regular graph is a graph with one vertex of degree $k-1$ and all other vertices of degree $k$.

\begin{theorem}{\rm (\cite{LLP})}\label{star union}
For integers $d_{1}\geq d_{2}\geq\cdots \geq d_{k}\geq1$, let $F=\cup_{1\leq i\leq k}K_{1,d_{i}}$. 
Then for large $n$, ${\rm ex}(n,F)=\max_{1\leq i\leq k}\left\{\binom{i-1}{2}+(i-1)(n-i+1)+\left\lfloor\frac{d_{i}-1}{2}(n-i+1)\right\rfloor\right\}$. Moreover, each extremal graph is of the form $K_{i-1}
\vee G_{n-i+1,d_{i}}$, where $1\leq i\leq k$ and $G_{n-i+1,d_{i}}$ is a (or nearly) $(d_{i}-1)$-regular graph.
\end{theorem}

\begin{theorem}{\rm (\cite{LLP})}\label{Path-star union}
For integers $a,b\geq1$, let $F=a\cdot P_{4}\cup b\cdot K_{1,3}$. Assume that $n=3d+r$ with $0\leq r\leq2$ is large. 
Then,\\
$(i)$ For $a=1$, then $K_{b}\vee (K_{r}\cup d\cdot K_{3})$ is the unique extremal graph for $F$ when $r=0$;
 $K_{b}\vee (K_{r}\cup d\cdot K_{3})$ and $K_{b+2a-1}\vee\overline{K_{n-b-2a+1}}$  are the only extremal graphs for $F$ when $r\neq0$.\\
$(ii)$ For $a\geq2$, then $K_{b+2a-1}\vee\overline{K_{n-b-2a+1}}$ is the unique extremal graph for $F$.
\end{theorem}

At the end of \cite{LLP}, the authors stated that "The same technique can be applied to $F=a\cdot P_{\ell}\cup b\cdot K_{1,t}$, but the proof is very technical". In this paper, we characterize the extremal graphs for all the remained kinds of disjoint unions of paths and stars. Our proofs contain new elements. These results can be described  in the following three theorems. (Recall that $P_{2}=K_{1,1}$ and $P_{3}=K_{1,2}$.)

\begin{theorem}\label{path-star th1}
For integers $a_{1}\geq a_{2}\geq\cdots\geq a_{t}\geq3$ and $\ell_{1}\geq \ell_{2}\geq\cdots\geq \ell_{r}\geq2$, let $F=\left(\cup_{1\leq i\leq t}K_{1,a_{i}}\right)\cup (\cup_{1\leq j\leq r}P_{\ell_{j}})$, where $t\geq1,r\geq2$. Set $b=\sum_{1\leq j\leq r}\left\lfloor\frac{\ell_{j}}{2}\right\rfloor$. Then for large $n$, each graph in ${\rm EX}(n,F)$ is an extremal graph for $\cup_{1\leq i\leq t}K_{1,a_{i}}$ (see Theorem \ref{star union}), or is one of the following graphs:\\
$(i)$ the graph obtained from $K_{t+b-1}\vee\overline{ K_{n-t-b+1}}$ by adding a maximum matching in the part $\overline{ K_{n-t-b+1}}$ when $\ell_{i}=3$ for any $1\leq i\leq r$; \\
$(ii)$  $K_{t+b-1}\vee(K_{2}\cup \overline{ K_{n-t-b-1}})$ when all $\ell_{i}$ are odd and at least one of them is not 3;\\
 $(iii)$ $K_{t+b-1} \vee\overline{K_{n-t-b+1}}$ when at least one of  $\ell_{i}$ with $1\leq i\leq r$ is even.
\end{theorem}

\begin{theorem}\label{path-star th2}
For integers $a_{1}\geq a_{2}\geq\cdots\geq a_{t}\geq3$, let
$F=(\cup_{1\leq i\leq t}K_{1,a_{i}})\cup P_{2k+2}$ with $k\geq1$. Then for large $n$, each graph in ${\rm EX}(n,F)$ is an extremal graph for $\cup_{1\leq i\leq t}K_{1,a_{i}}$, or is of the form
$K_{t}\vee H_{n-t,k}$, where $H_{n-t,k}$ satisfies the followings:\\
$(i)$ for $a_{t}\leq2k$, $H_{n-t,k}=K_{k}\vee(n-t-k)K_{1}$;\\
$(ii)$ for $a_{t}\geq2k+1$, $H_{n-t,k}$ is an extremal graph  of order $n-t$ for $P_{2k+2}$ (see Theorem \ref{ex n Pt sharp}). Additionally, $H_{n-t,k}=K_{k}\vee(n-t-k)\cdot K_{1}$, or $H_{n-t,k}$ has maximum degree at most $a_{t}-1$ when $n-t\equiv k~or~k+1(\mod 2k+1)$.
\end{theorem}

\begin{theorem}\label{path-star th3}
For integers $a_{1}\geq a_{2}\geq\cdots\geq a_{t}\geq3$, let $F=\cup_{1\leq i\leq t}K_{1,a_{i}}\cup P_{2k+3}$ with $k\geq1$. Then for large $n$, each graph in ${\rm EX}(n,F)$ is an extremal graph for $\cup_{1\leq i\leq t}K_{1,a_{i}}$, or is of the form
$K_{t}\vee H_{n-t,k}$, where $H_{n-t,k}$ satisfies the followings:\\
$(i)$ for $a_{t}\leq2k$, $H_{n-t,k}=K_{k}\vee(K_{2}\cup\overline{K_{n-t-k-2}})$;\\
$(ii)$ for $a_{t}=2k+1$, $H_{n-t,k}$ is a $2k$-regular and $P_{2k+3}$-free graph of order $n-t$, or $H_{n-t,k}=K_{1}\vee(K_{2}\cup\overline{K_{n-t-3}})$ when $k=1$;\\
$(iii)$ for $a_{t}\geq2k+2$, $H_{n-t,k}$ is an extremal graph for $P_{2k+3}$ of order $n-t$.
\end{theorem}

The rest of this paper is organized as follows. In Section 2, we give some lemmas needed in our proofs of the main results. In Section 3, we give the proofs of Theorems \ref{path-star th1}, \ref{path-star th2} and \ref{path-star th3}.

\section{Some lemmas}

The following lemma is a variation of Lemma 4 of \cite{LLP}, as a special case of
Lemma 2.3 of \cite{BuK}.

\begin{lem}\label{bipar lem}
 Let $1\leq a\leq s,v\geq2$ be given integers and $n\geq3v(s-a+1)\binom{s}{a}$. If $G$ is a graph
of order $n$ and with a set $S$ of $s$ vertices such that $e_{G}(S,V(G)-S)\geq(a-\frac{2}{3})n$,
then $S$ contains a subset of $a$ vertices with a common neighborhood of $v$ vertices in $V(G)-S$.
\end{lem}

\f{\bf Proof:} Let $n_{0}$ be the number of vertices in $V(G)-S$ which have at least $a$ neighbors in $S$. Then $e_{G}(S,V(G)-S)\leq n_{0} s+(n-s-n_{0})(a-1)$. Since $e_{G}(S,V(G)-S)\geq(a-\frac{2}{3})n$,  we obtain $n_{0}\geq\frac{n}{3(s-a+1)}$. Since there are $\binom{s}{a}$ sets of $a$ vertices in $S$, we see that $S$ must contain a subset of $a$ vertices with a common neighborhood of order at least $\frac{n_{0}}{\binom{s}{a}}\geq v$ in $V(G)-S$. \hfill$\Box$

\medskip

Let ${\rm ex_{con}}(n,P_{\ell})$ denote the maximum size of connected $P_{\ell}$-free graphs of order $n$.  Kopylov \cite{K} and Balister
et. al. \cite{BGLS}, independently, proved the following theorem.

\begin{lem}\label{connected P}
For integers $n\geq\ell\geq4$, 
$${\rm ex_{con}}(n,P_{\ell})=\max\left\{\binom{\ell-2}{2}+(n-\ell+2),
\binom{\left\lfloor\frac{\ell}{2}\right\rfloor-1}{2}
+\left(\left\lfloor\frac{\ell}{2}\right\rfloor-1\right)\left(n-\left\lfloor\frac{\ell}{2}\right\rfloor
+1\right)+c\right\},$$
where $c=1$ if $\ell$ is odd and $c=0$ otherwise. Moreover, the (connected) extremal graphs are $K_{1}\vee\left(K_{\ell-3}\cup\overline{K_{n-\ell+2}}\right)$, or $K_{\frac{\ell}{2}-1}
\vee\overline{K_{n-\frac{\ell}{2}+1}}$ for even $\ell\geq4$ and
  $K_{\frac{\ell-3}{2}}
\vee\left(K_{2}\cup \overline{K_{n-\frac{\ell+1}{2}}}\right)$ for odd $\ell\geq5$.
\end{lem}

When $n\geq2\ell$ in Lemma \ref{connected P}, a simple calculation shows that the connected extremal graph  is $K_{\frac{\ell}{2}-1}
\vee\overline{K_{n-\frac{\ell}{2}+1}}$ for even $\ell\geq4$, and $K_{\frac{\ell-3}{2}}
\vee\left(K_{2}\cup \overline{K_{n-\frac{\ell+1}{2}}}\right)$ for odd $\ell\geq5$. The following lemma is very important for the proofs of the main results of this paper.

\begin{lem}\label{main-lem}
Let $G$ be a graph and $S\subseteq V(G)$ with $|S|\geq1$. Assume that $Q$ is a component of $G-S$ which is $P_{\ell}$-free, where $\ell\geq4$. If the end vertices of any $P_{\ell-2}$ of $Q$ have no neighbors in $S$, then $e_{G}(S,V(Q))+e(Q)\leq(|S|+\frac{\ell}{2}-2)|Q|$.
\end{lem}

\f{\bf Proof:} Recall the assumption in the lemma: any end vertex of a $P_{\ell-2}$ of $Q$ has no neighbors in $S$. We first show the following two claims.

\medskip

\f{\bf Claim 1.} Let $Q_{0}$ be a connected induced  subgraph  of $Q$. If $Q_{0}$ contains no $P_{\ell-2}$ or contains a Hamilton cycle, then $e_{G}(S,V(Q_{0}))+e(Q_{0})\leq(|S|+\frac{\ell}{2}-2)|Q_{0}|$.

\medskip

\f{\bf Proof of  Claim 1.} If $Q_{0}$ contains no $P_{\ell-2}$, then $e(Q_{0})\leq(\frac{\ell}{2}-2)|Q_{0}|$ by Theorem \ref{ex n Pt}. It follows that
$e_{G}(S,V(Q_{0}))+e(Q_{0})\leq(|S|+\frac{\ell}{2}-2)|Q_{0}|$, as desired.
 
It remains to consider the case that $Q_{0}$ contains Hamilton cycles. If $|Q_{0}|\leq\ell-3$, then  clearly $e(Q_{0})\leq(\frac{\ell}{2}-2)|Q_{0}|$. It follows that $e_{G}(S,V(Q_{0}))+e(Q_{0})\leq(|S|+\frac{\ell}{2}-2)|Q_{0}|$, as desired.
Now we can assume that $|Q_{0}|\geq\ell-2$. Then each vertex of $Q_{0}$ can be viewed as an end vertex of a $P_{\ell-2}$ in $Q_{0}$, since $Q_{0}$ contains Hamilton cycles. By assumption,
we have $e_{G}(S,V(Q_{0}))=0$. Note that $e(Q_{0})\leq(\frac{\ell}{2}-1)|Q_{0}|$ by Theorem \ref{ex n Pt}, since $Q_{0}$ is $P_{\ell}$-free. Thus, $e_{G}(S,V(Q_{0}))+e(Q_{0})\leq(\frac{\ell}{2}-1)|Q_{0}|\leq(|S|+\frac{\ell}{2}-2)|Q_{0}|$ as $|S|\geq1$.
This finishes the proof of Claim 1. \hfill$\Box$

\medskip

\medskip

\f{\bf Claim 2.} Let $Q_{0}$ be a connected induced  subgraph  of $Q$.  If  $Q_{0}$ contains a $P_{\ell-2}$ and contains no Hamilton cycle, then there is a vertex $u\in V(Q_{0})$ such that $Q_{0}-u$ is connected. Moreover, $d_{Q_{0}}(u)\leq \frac{\ell}{2}-1$ and $e_{G}(u,S)=0$.

\medskip

\f{\bf Proof of  Claim 2.} Let $w_{1}w_{2}\cdots w_{a}$ be a longest path in $Q_{0}$. Since $Q$ is $P_{\ell}$-free, we have $a\leq\ell-1$. Note that $a\geq\ell-2$, since $Q_{0}$ contains a $P_{\ell-2}$. Thus, both $w_{1}$ and $w_{a}$ have no neighbors in $S$. Let $H$ be the subgraph of $Q_{0}$ induced by $w_{1},w_{2},..., w_{a}$. Since $Q_{0}$ has no Hamilton cycle and $w_{1}w_{2}\cdots w_{a}$ is a longest path in $Q_{0}$, we see that $H$ has no Hamilton cycle and $w_{1},w_{a}$ have no neighbor in $V(Q_{0})-V(H)$.  Note that $w_{1},w_{2},..., w_{a}$ is a Hamilton path of $H$. Set $X=\left\{w_{i}~|~w_{1}\sim w_{i+1},1\leq i\leq a-1\right\}$
and $Y=\left\{w_{j}~|~w_{a}\sim w_{j},1\leq j\leq a-1\right\}$. Clearly, $|X|=d_{H}(w_{1})$ and $|Y|=d_{H}(w_{a})$.
Since $H$ has no Hamilton cycle, we must have $X\cap Y=\emptyset$. In fact, if $w_{i_{0}}\in X\cap Y$ for some $1\leq i_{0}\leq a-1$, then $w_{1}w_{i_{0}+1}w_{i_{0}+2}\cdots w_{a}w_{i_{0}}w_{i_{0}-1}\cdots w_{1}$ will be a Hamilton cycle of $H$,  a contradiction. It follows that $d_{H}(w_{1})+d_{H}(w_{a})=|X|+|Y|\leq a-1$. Thus we can assume that $d_{H}(w_{1})\leq\frac{a-1}{2}\leq\frac{\ell-2}{2}$ without loss of generality. Since $Q_{0}$   is connected and $w_{1}$ has no neighbor in $V(Q_{0})-V(H)$, we must have $d_{Q_{0}}(w_{1})=d_{H}(w_{1})\leq\frac{\ell-2}{2}$, and $Q_{0}-w_{1}$ is connected. Recall that $w_{1}$ has no neighbor in $S$. Thus, Claim 2 is completed by letting $u=w_{1}$. \hfill$\Box$

\medskip

 If $Q$ contains no $P_{\ell-2}$ or contains a Hamilton cycle, then by Claim 1, $e_{G}(S,V(Q))+e(Q)\leq(|S|+\frac{\ell}{2}-2)|Q|$.
Otherwise, $Q$ contains a  $P_{\ell-2}$ and contains no Hamilton cycle. Then by Claim 2 (letting $Q_{0}=Q)$, there is a vertex $u_{1}\in V(Q)$ such that $Q-u_{1}$ is connected. Moreover, $d_{Q}(u_{1})\leq \frac{\ell}{2}-1$ and $e_{G}(u_{1},S)=0$. Set $Q_{1}=Q-u_{1}$.
If $Q_{1}$ contains no $P_{\ell-2}$ or contains a Hamilton cycle, then by Claim 1, $e_{G}(S,V(Q_{1}))+e(Q_{1})\leq(|S|+\frac{\ell}{2}-2)|Q_{1}|$.
Otherwise, $Q_{1}$ contains a $P_{\ell-2}$ and contains no Hamilton cycle. Then by Claim 2, there is a vertex $u_{2}\in V(Q_{1})$ such that $Q_{1}-u_{2}$ is connected. Moreover, $d_{Q_{1}}(u_{2})\leq \frac{\ell}{2}-1$ and $e_{G}(u_{2},S)=0$. Set $Q_{2}=Q_{1}-u_{2}$.
We do a similar discussion for $Q_{2}$ as above. Eventually, we can obtain a series of vertices of $Q$: $u_{1},u_{2},...,u_{m}$ for some integer $m$, such that $d_{Q_{i-1}}(u_{i})\leq \frac{\ell}{2}-1$ and $e_{G}(u_{i},S)=0$ for any $1\leq i\leq m$, where $Q_{i}=Q-\left\{u_{1},u_{2},...,u_{i}\right\}$ and $Q_{0}=Q$. Moreover, $Q_{m}$ contains no $P_{\ell-2}$ or contains a Hamilton cycle. Then by Claim 1, $e_{G}(S,V(Q_{m}))+e(Q_{m})\leq(|S|+\frac{\ell}{2}-2)|Q_{m}|$. Note that $|Q|=m+|Q_{m}|$.
 It follows that
 \begin{equation}
\begin{aligned}
&e_{G}(S,V(Q))+e(Q)\\
&\leq e_{G}(S,V(Q_{m}))+e(Q_{m})+\sum_{1\leq i\leq m}d_{Q_{i-1}}(u_{i})\\
&\leq(|S|+\frac{\ell}{2}-2)|Q_{m}|+(\frac{\ell}{2}-1)m\\
&\leq(|S|+\frac{\ell}{2}-2)|Q|.
\end{aligned}\notag
\end{equation}
This completes the proof. \hfill$\Box$

\section{Proofs of Theorems \ref{path-star th1}, \ref{path-star th2} and \ref{path-star th3}}

In this section, we will give the proofs of Theorems \ref{path-star th1}, \ref{path-star th2} and \ref{path-star th3}.

\medskip

\f{\bf Proof of Theorem \ref{path-star th1}.} Assume that $G$ is an extremal graph for $F$. Clearly, $K_{t+b-1}\vee\overline{K_{n-t-b+1}}$ is $F$-free. To see this, observe that each $K_{1,a_{j}}$ must contain a vertex in the part $K_{t+b-1}$. Thus, if $K_{t+b-1}\vee\overline{K_{n-t-b+1}}$ contains a $F$, then $K_{b-1}\vee\overline{K_{n-t-b+1}}$ must contain a $\cup_{1\leq j\leq r}P_{\ell_{j}}$, a contradiction (see Theorem \ref{Path union}).
Thus, $K_{t+b-1}\vee\overline{K_{n-t-b+1}}$ is $F$-free. So,
$$e(G)\geq e(K_{t+b-1}\vee\overline{K_{n-t-b+1}})\geq(t+b-1)(n-t-b+1).$$
If $G$ does not contain a $\cup_{1\leq i\leq t}K_{1,a_{i}}$, then $G$ must be an extremal graph for $\cup_{1\leq i\leq t}K_{1,a_{i}}$, as desired. Thus, we can assume that $G$ contains a $\cup_{1\leq i\leq t}K_{1,a_{i}}$. Let $S'$ be the set of vertices of a  $\cup_{1\leq i\leq t}K_{1,a_{i}}$ in $G$. Then $G-S'$ contains no $\cup_{1\leq j\leq r}P_{\ell_{j}}$, otherwise $G$ will contain an $F$,  a contradiction. Then $e(G-S')\leq(b-1)n+1$ by Theorem \ref{Path union}. Thus, 
$$e_{G}(S',V(G)-S')\geq(t+b-1)(n-t-b+1)-\binom{|S'|}{2}-(b-1)n-1\geq(t-\frac{2}{3})n$$ for large $n$.
By Lemma \ref{bipar lem}, there is a $S\subseteq S'$ with $|S|=t$, such that the vertices in $S$ have a common neighborhood of size $v$ for a large integer $v>3b+2\sum_{1\leq i\leq t}a_{i}$ (requiring $n$ large).

Now we show that $G-S$ contains no $\cup_{1\leq j\leq r}P_{\ell_{j}}$. Suppose that $G-S$  contains a $\cup_{1\leq j\leq r}P_{\ell_{j}}$.  We can obtain an $F$ in $G$ by choosing a $\cup_{1\leq i\leq t}K_{1,a_{i}}$ with centers at the vertices in $S$ and with leaves not in the set of the vertices of a $\cup_{1\leq j\leq r}P_{\ell_{j}}$ in $G-S$. We can do this, since $v$ is large.  Thus, $G-S$ contains no $\cup_{1\leq j\leq r}P_{\ell_{j}}$.

Recall $r\geq2$. One can check that the graphs in $(i),(ii),(iii)$ are $F$-free (see Theorems \ref{kP3} and \ref{Path union}). 
Recall that $G-S$ contains no $\cup_{1\leq j\leq r}P_{\ell_{j}}$. Since $G$ is extremal for $F$,  we must have that $G-S$ is an extremal graph for $\cup_{1\leq j\leq r}P_{\ell_{j}}$. Consequently, $G$ must be one of the graphs described in $(i),(ii),(iii)$.
This completes the proof. \hfill$\Box$

\medskip

\f{\bf Proof of Theorem \ref{path-star th2}.} Assume that $G$ is an extremal graph for $F$. Clearly, $K_{t+k}\vee\overline{K_{n-t-k}}$ is $F$-free. Thus, 
$$e(G)\geq(t+k)(n-t-k).$$
If $G$ does not contain a $\cup_{1\leq i\leq t}K_{1,a_{i}}$, then $G$ must be an extremal graph for $\cup_{1\leq i\leq t}K_{1,a_{i}}$, as desired. Thus, we can assume that $G$ contains a $\cup_{1\leq i\leq t}K_{1,a_{i}}$. Let $S'$ be the set of vertices of a  $\cup_{1\leq i\leq t}K_{1,a_{i}}$ in $G$. Then $G-S'$ contains no $P_{2k+2}$, otherwise $G$ will contain an $F$,  a contradiction. Then $e(G-S')\leq kn$ by Theorem \ref{ex n Pt}. Thus, 
$$e_{G}(S',V(G)-S')\geq(t+k)(n-t-k)-\binom{|S'|}{2}-kn\geq(t-\frac{2}{3})n$$ for large $n$.
By Lemma \ref{bipar lem}, there is a $S\subseteq S'$ with $|S|=t$, such that the vertices in $S$ have a common neighborhood of size $v$ for a large integer $v>2k+2+2\sum_{1\leq i\leq t}a_{i}$ (requiring $n$ large). Using a similar discussion as Theorem \ref{path-star th1}, 
we can show that $G-S$ contains no $P_{2k+2}$. 
Set $\ell=2k+2$ and $S=\left\{u_{1},u_{2},...,u_{t}\right\}$. 
Let $Q_{1},Q_{2},...,Q_{m}$ be all the components of $G-S$, where $m\geq1$. Then each $Q_{i}$ contains no $P_{\ell}$.

$(i)$ $a_{t}\leq2k$. This implies $k\geq2$ as $a_{t}\geq3$. Clearly, $K_{t+k}\vee(n-t-k)K_{1}$ is $F$-free. Thus, $$e(G)\geq\binom{t}{2}+t(n-t)+\binom{k}{2}+k(n-t-k).$$
If all components of $G-S$ have maximum degree at most $2k-1$, then $e(G-S)\leq\frac{2k-1}{2}(n-t)$. Thus, $e(G)\leq\binom{t}{2}+t(n-t)+\frac{2k-1}{2}(n-t)<\binom{t}{2}+t(n-t)+\binom{k}{2}+k(n-t-k)$, a contradiction. So, without loss of generality, we can assume that $Q_{1}$ has maximum degree at least $2k$. Thus, $Q_{1}$ contains a $K_{1,a_{t}}$ as $a_{t}\leq 2k$. 

Now we show that for each other component $Q_{i}$ with $2\leq i\leq m$,
each end vertex of a $P_{\ell-2}$ of $Q_{i}$ has no neighbors in $S$. Suppose that $w_{1}w_{2}\cdots w_{\ell-2}$ is a path in $Q_{i}$ with $i\geq2$ such that $w_{1}$ has a neighbor in $S$, say $u_{t}$. Recall that $v$ is large. Choose a  neighbor of $u_{t}$ in $V(G)-S$, say $y$, distinct from $w_{1},w_{2},..., w_{\ell-2}$ and the vertices in $K_{1,a_{t}}$ in $Q_{1}$. Then $yu_{t}w_{1}w_{2}\cdots w_{\ell-2}$ is a $P_{2k+2}$. For $1\leq j\leq t-1$, we can find a star $K_{1,a_{j}}$ centered at $u_{j}$ with leaves not used. Then $G$ contains an $F$, a contradiction. 

Recall that each end vertex of a $P_{\ell-2}$ of $Q_{i}$ with $i\geq2$ has no neighbors in $S$. By Lemma \ref{main-lem},  $e_{G}(S,V(Q_{i}))+e(Q_{i})\leq(|S|+\frac{\ell}{2}-2)|Q_{i}|=(|S|+k-1)|Q_{i}|$ for any $2\leq i\leq m$. Set $Q_{\geq2}=\cup_{2\leq j\leq m}Q_{j}$. Then $e_{G}(S,V(Q_{\geq2}))+e(Q_{\geq2})\leq(|S|+k-1)|Q_{\geq2}|$. Set $n_{1}=|Q_{1}|$.

If $n_{1}\leq6k$, then $e(Q_{1})\leq kn_{1}\leq6k^{2}$ by Lemma \ref{ex n Pt}.
Thus,
 \begin{equation}
\begin{aligned}
e(G)&= e(G[S])+e_{G}(S,V(Q_{1}))+e_{G}(S,V(Q_{\geq2}))+e(Q_{\geq2})+e(Q_{1})\\
&\leq e(G[S])+|S|n_{1}+(|S|+k-1)|Q_{\geq2}|+e(Q_{1})\\
&\leq\binom{t}{2}+t(n-t)+(k-1)(n-t-n_{1})+6k^{2}\\
&<\binom{t}{2}+t(n-t)+\binom{k}{2}+k(n-t-k),
\end{aligned}\notag
\end{equation}
 a contradiction.

If $n_{1}\geq6k$, then $e(Q_{1})\leq \binom{k}{2}+k(n_{1}-k)$ by Lemma \ref{connected P}, since $Q_{1}$ is connected and $P_{2k+2}$-free.
Set $n_{i}=|Q_{i}|$.
Thus, 
\begin{equation}
\begin{aligned}
e(G)&= e(G[S])+e_{G}(S,V(Q_{1}))+e_{G}(S,V(Q_{\geq2}))+e(Q_{\geq2})+e(Q_{1})\\
&\leq e(G[S])+|S|n_{1}+(|S|+k-1)|Q_{\geq2}|+e(Q_{1})\\
&\leq\binom{t}{2}+t(n-t)+(k-1)(n-t-n_{1})+\binom{k}{2}+k(n_{1}-k)\\
&=\binom{t}{2}+t(n-t)+\binom{k}{2}+k(n-t-k)-(n-t-n_{1})\\
&\leq\binom{t}{2}+t(n-t)+\binom{k}{2}+k(n-t-k).
\end{aligned}\notag
\end{equation}
Recall that $e(G)\geq\binom{t}{2}+t(n-t)+\binom{k}{2}+k(n-t-k)$. Thus, equality holds. It must be that $n=t+n_{1}$ and $G-S=K_{k}\vee(n-t-k)K_{1}$. Moreover, $G=K_{t+k}\vee(n-t-k)K_{1}$, as desired.

 $(ii)$ $a_{t}\geq 2k+1$. Let $G'=K_{t}\vee H_{n-t,k}$, where $H_{n-t,k}$ is an extremal graph for $P_{2k+2}$ of order $n-t$. Additionally, $H_{n-t,k}=K_{k}\vee(n-t-k)\cdot K_{1}$, or $H_{n-t,k}$ has maximum degree at most $a_{t}-1$ when $n-t\equiv k~or~k+1(\mod 2k+1)$. Since $a_{t}\geq 2k+1$, $G'$ is $F$-free. Recall that $G-S$ is $P_{2k+2}$-free. Thus, $e(G)\leq e(G')$. It must be that $e(G)=e(G')$, and $G$ is of the form described in the theorem.
This completes the proof. \hfill$\Box$

\medskip

\begin{theorem}
$\cup_{1\leq i\leq t}K_{1,a_{i}}\cup P_{2k+3}$ with $k\geq1$.
$K_{t}\vee H_{n-t,k}$, where $H_{n-t,k}$ satisfies the followings:\\
$(i)$ for $k\geq2$ and $a_{t}\leq2k$, $H_{n-t,k}$ is the graph obtained from $K_{k}\vee(n-t-k)K_{1}$ by adding one edge;\\
$(ii)$ for $k\geq1$ and $a_{t}=2k+1$, $H_{n-t,k}$ is $2k$-regular  $P_{2k+3}$-free graph of order $n-t$, or the graph obtained from $K_{1,n-t-1}$ by adding one edge when $k=1$ and $a_{t}=3$;\\
$(iii)$ for $k\geq1$ and $a_{t}\geq2k+2$, $H_{n-t,k}$ is an extremal graph for $P_{2k+3}$ of order $n-t$.
\end{theorem}

\f{\bf Proof of Theorem \ref{path-star th3}.} Assume that $G$ is an extremal graph for $F$. Clearly, $K_{t+k}\vee\overline{K_{n-t-k}}$ is $F$-free. Thus, 
$$e(G)\geq(t+k)(n-t-k).$$
If $G$ does not contain a $\cup_{1\leq i\leq t}K_{1,a_{i}}$, then $G$ must be an extremal graph for $\cup_{1\leq i\leq t}K_{1,a_{i}}$, as desired. Thus, we can assume that $G$ contains a $\cup_{1\leq i\leq t}K_{1,a_{i}}$. Using a similar discussion as Theorem \ref{path-star th2}, we can obtain a $S\subseteq V(G)$ with $|S|=t$, such that the vertices in $S$ have a large common neighborhood in $V(G)-S$, and $G-S$ contains no $P_{2k+3}$. 
Set $\ell=2k+3$ and $S=\left\{u_{1},u_{2},...,u_{t}\right\}$. 
Let $Q_{1},Q_{2},...,Q_{m}$ be all the components of $G-S$, where $m\geq1$. Then each $Q_{i}$ contains no $P_{\ell}$. Since $K_{t+k}\vee(K_{2}\cup\overline{K_{n-t-k+2}})$ is $F$-free, we have $$e(G)\geq\binom{t}{2}+t(n-t)+\binom{k}{2}+k(n-t-k)+1.$$

$(i)$ $a_{t}\leq2k$. This implies $k\geq2$.
If each component of $G-S$ has maximum degree at most $2k-1$, then $e(G-S)\leq\frac{2k-1}{2}(n-t)$. Thus, $$e(G)\leq\binom{t}{2}+t(n-t)+\frac{2k-1}{2}(n-t)<\binom{t}{2}+t(n-t)+\binom{k}{2}+k(n-t-k)+1,$$
 a contradiction. So, without loss of generality, we can assume that $Q_{1}$ has maximum degree at least $2k$. Thus, $Q_{1}$ contains a $K_{1,a_{t}}$ as $a_{t}\leq 2k$. Similar to Theorem \ref{path-star th2}, for each other component $Q_{i}$ with $2\leq i\leq m$,
each end vertex of a $P_{\ell-2}$ of $Q_{i}$ has no neighbors in $S$. 
By Lemma \ref{main-lem},  $e_{G}(S,V(Q_{i}))+e(Q_{i})\leq(|S|+\frac{\ell}{2}-2)|Q_{i}|=(|S|+\frac{2k-1}{2})|Q_{i}|$ for any $2\leq i\leq m$. Set $n_{1}=|Q_{1}|$ and $Q_{\geq2}=\cup_{2\leq j\leq m}Q_{j}$. Then $e_{G}(S,V(Q_{\geq2}))+e(Q_{\geq2})\leq(|S|+\frac{2k-1}{2})|Q_{\geq2}|$.

If $n_{1}\leq6k$, then $e(Q_{1})\leq \frac{2k+1}{2}n_{1}\leq3k(2k+1)$ by Lemma \ref{ex n Pt}.
Thus, $$e(G)\leq\binom{t}{2}+t(n-t)+\frac{2k-1}{2}(n-t-n_{1})+3k(2k+1)
<\binom{t}{2}+t(n-t)+\binom{k}{2}+k(n-t-k)+1,$$ a contradiction. 

If $n_{1}\geq6k$, then $e(Q_{1})\leq \binom{k}{2}+k(n_{1}-k)+1$ by Lemma \ref{connected P}, since $Q_{1}$ is connected and $P_{\ell}$-free.
Thus, 
$$e(G)\leq\binom{t}{2}+t(n-t)+\frac{2k-1}{2}(n-t-n_{1})+\binom{k}{2}+k(n_{1}-k)+1.$$
One can check that if $n-t-n_{1}>0$, then
$$e(G)<\binom{t}{2}+t(n-t)+\binom{k}{2}+k(n-t-k)+1,$$
 a contradiction. Thus, $n-t-n_{1}=0$ and consequently $G=K_{t+k}\vee(K_{2}\cup\overline{K_{n-t-k-2}})$.

$(ii)$ $a_{t}=2k+1$. Let $G'=K_{t}\vee H_{n-t,k}$, where $H_{n-t,k}$ is a $2k$-regular and  $P_{2k+3}$-free graph of order $n-t$. Clearly, $G'$ is $F$-free as $a_{t}=2k+1$. Thus,
$$e(G)\geq e(G')=\binom{t}{2}+t(n-t)+k(n-t).$$ 
If each component of $G-S$ has maximum degree at most $2k$, then $e(G-S)\leq k(n-t)$. Thus, $e(G)\leq\binom{t}{2}+t(n-t)+k(n-t)$. Then equality holds, and it must be that $G$ is of the same form as $G'$, as desired.  So, without loss of generality, it remains to consider  that $Q_{1}$ has maximum degree at least $2k+1$. Thus, $Q_{1}$ contains a $K_{1,a_{t}+1}$ as $a_{t}= 2k+1$. Then, similar to $(i)$, for each other component $Q_{i}$ with $2\leq i\leq m$,
each end vertex of a $P_{\ell-2}$ of $Q_{i}$ has no neighbors in $S$. 
By Lemma \ref{main-lem},  $$e_{G}(S,V(Q_{i}))+e(Q_{i})\leq(|S|+\frac{\ell}{2}-2)|Q_{i}|=(|S|+\frac{2k-1}{2})|Q_{i}|$$
 for any $2\leq i\leq m$. Set $n_{1}=|Q_{1}|$ and $Q_{\geq2}=\cup_{2\leq j\leq m}Q_{j}$. Then $$e_{G}(S,V(Q_{\geq2}))+e(Q_{\geq2})\leq(|S|+\frac{2k-1}{2})|Q_{\geq2}|.$$

If $n_{1}\leq6k$, then $e(Q_{1})\leq \frac{2k+1}{2}n_{1}\leq3k(2k+1)$ by Lemma \ref{ex n Pt}.
Thus, $$e(G)\leq\binom{t}{2}+t(n-t)+\frac{2k-1}{2}(n-t-n_{1})+3k(2k+1)
<\binom{t}{2}+t(n-t)+k(n-t),$$ a contradiction.

If $n_{1}\geq6k$, then $e(Q_{1})\leq \binom{k}{2}+k(n_{1}-k)+1$ by Lemma \ref{connected P}, since $Q_{1}$ is connected and $P_{\ell}$-free.
Thus, $e(G)\leq\binom{t}{2}+t(n-t)+\frac{2k-1}{2}(n-t-n_{1})+\binom{k}{2}+k(n_{1}-k)+1
\leq\binom{t}{2}+t(n-t)+k(n-t)$. We must have equality, and $n-t-n_{1}=0$ and $k=1$. Consequently, $G=K_{t+1}\vee(K_{2}\cup\overline{K_{n-t-3}})$, as desired.

$(iii)$ $a_{t}\geq2k+2$. Let $G''=K_{t}\vee H'(n-t,k)$, where $H'(n-t,k)$ is an extremal graph of order $n-t$ for $P_{2k+3}$. Note that $H'(n-t,k)$ has maximum degree at most $2k+1<a_{t}$ by Lemma \ref{ex n Pt sharp}. Thus, $G''$ is $F$-free. Recall that $G-S$ is $P_{2k+3}$-free. Thus, $G$ must be of the same form as $G''$, as desired. 
This completes the proof. \hfill$\Box$

\medskip

\f{\bf Declaration of competing interest}

\medskip

There is no conflict of interest.

\medskip

\f{\bf Data availability statement}

\medskip

No data was used for the research described in the article.


\medskip

\begin{thebibliography}{99}



\bibitem{BGLS}
P. Balister, E. Gy\H{o}ri, J. Lehel and R. Schelp, Connected graphs without long paths, Discrete Math. 308 (2008) 4487-4494.

\bibitem{BieK}
H. Bielak and S. Kieliszek, The Tur\'{a}n number of the graph $2P_{5}$, Discuss. Math. Graph Theory. 36 (2016) 683-694.

\bibitem{B}
 B. Bollob\'{a}s, Extremal Graph Theory, Academic Press, New York, 1978.




\bibitem{BuK}
 N. Bushaw and N. Kettle. Tur\'{a}n numbers of multiple paths and equibipartite forests,
Combin. Probab. Comput. 20 (2011) 837-853.




\bibitem{EG}
 P. Erd\H{o}s and T. Gallai, On maximal paths and circuits of graphs, Acta Math. Hungar. 10 (3) (1959) 337-356.






\bibitem{E Simonovits}
 P. Erd\H{o}s and M. Simonovits, A limit theorem in graph theory, Studia Sci. Math. Hungar. 1 (1966) 51-57.
 
\bibitem{E Stone}
 P. Erd\H{o}s and A. Stone, On the structure of linear graphs, Bull. Am. Math. Soc. 52 (1946) 1087-1091.

\bibitem{FS}
R. Faudree and R. Schelp, Path Ramsey numbers in multicolorings, J. Combin. Theory B 19 (1975) 150160.

\bibitem{G}
 I. Gorgol, Tur\'{a}n numbers for disjoint copies of graphs, Graphs Comb. 27 (2011) 661-667.





\bibitem{K}
G. Kopylov, On maximal paths and cycles in a graph, Soviet Math. Dokl. 18 (1977) 593-596.
 
 \bibitem{KST}
 T. K\H{o}v\'{a}ri, V. S\'{o}s and P. Tur\'{a}n. On a problem of K. Zarankiewicz, Colloquium
Math. 3 (1954) 50-57.



\bibitem{LLP}
 B. Lidick\'{y}, H. Liu and C. Palmer, On the Tur\'{a}n number of forests, Electron. J. Comb. 20 (2) (2013) P62.


\bibitem{LQS}
Y. Lan, Z. Qin and Y. Shi, The Tur\'{a}n number for $2P_{7}$, Discuss. Math. Graph Theory. 39 (2019) 805-814.









\bibitem{T}
 P. Tur\'{a}n, On an extremal problem in graph theory (in Hungrarian), Mat. Fiz. Lapok 48 (1941) 436-452.

\bibitem{YZ0}
L. Yuan and X. Zhang, The Tur\'{a}n number of disjoint copies of paths, Discrete Math. 340 (2017) 
132-139.

\bibitem{YZ}
 L. Yuan, and X. Zhang, Tur\'{a}n numbers for disjoint paths, J. Graph Theory 98 (2021) 499-524.







\end{thebibliography}
\end{document}